\documentclass[12pt]{amsart}
\pdfoutput=1
\usepackage[utf8]{inputenc}
\usepackage{MnSymbol}
\usepackage{accents}
\usepackage{tikz-cd}

\usepackage{bm}
\usepackage[textsize=footnotesize,shadow]{todonotes}
\usepackage{fancyhdr}
\usepackage[colorlinks=true,unicode=true]{hyperref}
\usepackage{mathtools}
\usepackage{lpic}

\input{style.tex}
\input{defs.tex}

\def\draft{\fbox{\tiny draft: \today}}
\let\draft=\relax
\setcolors[rmcolor=cyan,
           commentcolor=purple,
           inscolor=blue,
           commentsymbol=$\boxed{\mathbb{S}}$,
           rmsymbol=$\boxed{\mathbb{S}_{rm}}$,
           inssymbol=$\boxed{\mathbb{S}_{ins}}$]
          {s}
\setcolors[rmcolor=green,
           commentcolor=red,
           inscolor=brown,
           commentsymbol=$\boxed{\mathbb{J}}$,
           rmsymbol=$\boxed{\mathbb{J}_{rm}}$,
           inssymbol=$\boxed{\mathbb{J}_{ins}}$]
          {j}

\title{Conditioning in tropical probability theory}
\author[RM]{R. Matveev}
\author[JWP]{J. W. Portegies}
\begin{document}
\thispagestyle{fancy} 

\begin{abstract}
  We define a natural operation of conditioning of tropical diagrams
  of probability spaces and show that it is Lipschitz continuous with
  respect to the asymptotic entropy distance.
\end{abstract}

\maketitle

\son
\jon

\section{Introduction}
In ~\cite{Matveev-Asymptotic-2018}, ~\cite{Matveev-Tropical-2019} we have
initiated the study of tropical probability spaces and their diagrams:
In ~\cite{Matveev-Asymptotic-2018} we endowed (commutative) diagrams
of probability spaces with the intrinsic entropy distance and
in~\cite{Matveev-Tropical-2019} we defined tropical diagrams as points in
the asymptotic cone of the metric space. They are represented by
certain sequences of diagrams of probability spaces.

We expect that tropical diagrams will be helpful in the study of
information optimization problems, and we have indeed applied them to
derive a dimension-reduction result for the shape of the entropic cone
in \cite{Matveev-Tropical-Entropic-2019}.

In the present article we introduce the notion of conditioning on a
space in a tropical diagram and show that the operation is
Lipschitz-continuous with respect to the asymptotic entropy distance.

It is a rather technical result, and we have therefore decided to
treat it in this separate article, but it is an important ingredient
in the theory, and in particular we need it for the
dimension-reduction result mentioned before.

Given a tuple of finite-valued random variables $(\Xsf_{i})_{i=1}^{n}$
and a random variable $\Ysf$, one may ``condition'' the collection
$(\Xsf_{i})$ on $\Ysf$. The result of this operation is a family of
$n$-tuples of random variables denoted $(\Xsf_{i}\rel \Ysf)_{i=1}^{n}$
parameterized by those values of $\Ysf$ that have positive
probability. Each tuple of random variable in this family is defined
on a separate probability space.

When passing to the tropical setting the situation is different in the
sense that when we condition a tropical diagram $[\Xcal]$ on a space
$[Y]$, the result is again a tropical diagram $[\Xcal\rel Y]$ rather
than a family.  After recalling some preliminaries in
Section~\ref{se:prelim}, we describe the operation of conditioning and
prove that the result depends in a Lipschitz way on the original
diagram in Section~\ref{se:conditioning}.

\section{Preliminaries}
\label{se:prelim}
Our main objects of study are commutative diagrams of probability
spaces and their tropical counterparts. In this section we recall
briefly the main definitions and results.

\subsection{Probability spaces and their diagrams}
\subsubsection{Probability spaces}
By a \term{finite probability space} we mean a set with a
probability measure, that has finite support. A \term{reduction} from
one probability space to another is an equivalence class of
measure-preserving maps. Two maps are equivalent, if they coincide on a
set of full measure. We call a point $x$ in a probability space
$X=(\underline{X},p)$ an \term{atom} if it has positive weight and we
write $x\in X$ to mean $x$ is an atom in $X$ (as opposed to
$x\in\underline{X}$ for points in the underlying set). For a
probability space $X$ we denote by $|X|$ the cardinality of the
support of the probability measure.

\subsubsection{Indexing categories}
To record the combinatorial structure of a commutative diagrams of
probability spaces and reductions we use an object that we call an
\term{indexing category}. By an indexing category we mean a finite
category $\Gbf$ such that for any pair of objects $i,j\in\Gbf$ there
is at most one morphism between them either way. In addition, we will
assume it satisfies one additional property that we will describe
after introducing some terminology. For a pair of objects $i,j\in\Gbf$
such that there is a morphism $\gamma_{ij}:i\to j$, object $i$ will be
called an \term{ancestor} of $j$ and object $j$ will be called a
\term{descendant} of $i$. The subcategory of all descendants of an
object $i\in\Gbf$ is called an \term{ideal} generated by $i$ and will
be denoted $\lc i\rc$, while we will call the subcategory consisting
of all ancestors of $i$ together with all the morphisms in it a
\term{co-ideal} generated by $i$ and denote it by $\lf
i\rf$. (The term \term{filter} is also used for co-ideal in the
literature about lattices)

The additional property that an indexing
category has to satisfy is that for any pair of objects $i,j\in\Gbf$
there exists a \term{minimal common ancestor} $\hat\imath$, that is
$\hat\imath$ is an ancestor for both $i$ and $j$ and any other
ancestor of them both is also an ancestor of $\hat\imath$. 

An equivalent formulation of the property above is the following: the
intersection of the co-ideals generated by two objects
$i,j\in\Gbf$ is also a co-ideal generated by some object
$\hat\imath\in\Gbf$.

Any indexing category $\Gbf$ is necessarily \term{initial}, which
means that there exists an \term{initial object}, that is an object
$i_{0}$ such that $\Gbf=\lc i_{0}\rc$.

A \term{fan} in a category is a pair of morphisms with the same
domain.
A fan $(i\ot k\to j)$ is called
\term{minimal} if for any other fan $(i\ot l\to j)$ included
in a commutative diagram
\[
\begin{cd}[row sep=-1mm]
  \mbox{}
  \&
  k
  \arrow{dl}
  \arrow{dd}
  \arrow{dr}
  \\
  i
  \&\&
  j
  \\
  \&
  l
  \arrow{ul}
  \arrow{ur}
\end{cd}
\]
the vertical arrow must be an isomorphism.

For any pair of objects $i,j$ in an indexing category $\Gbf$ there
exists a unique minimal fan $(i\ot\hat\imath\to j)$ in $\Gbf$.

\subsubsection{Diagrams}
We denote by $\prob$ the category of finite probability spaces and
reductions. For an indexing category $\Gbf=\set{i;\, \gamma_{ij}}$, a
$\Gbf$-diagram is a functor $\Xcal:\Gbf\to\prob$. A reduction $f$ from one
$\Gbf$-diagram $\Xcal=\set{X_{i};\,\chi_{ij}}$ to another
$\Ycal=\set{Y_{i};\, \upsilon_{ij}}$ is a natural transformation
between the functors. It amounts to a collection of reductions
$f_{i}:X_{i}\to Y_{i}$ such that the big  diagram consisting
of all spaces $X_{i}$, $Y_{i}$ and all morphisms $\chi_{ij}$,
$\upsilon_{ij}$ and $f_{i}$ is commutative.  The category of
$\Gbf$-diagrams and reductions will be denoted $\prob\<\Gbf\>$.
The construction of diagrams could be iterated, thus we can consider
$\Hbf$-diagrams of $\Gbf$-diagrams and denote the corresponding
category $\prob\<\Gbf\>\<\Hbf\>=\prob\<\Gbf,\Hbf\>$. Every
$\Hbf$-diagram of $\Gbf$-diagrams can also be considered as
$\Gbf$-diagram of $\Hbf$-diagrams, thus there is a natural equivalence
of categories $\prob\<\Gbf,\Hbf\>\cong\prob\<\Hbf,\Gbf\>$.

A $\Gbf$-diagram $\Xcal$ will be called \term{minimal} if it maps
minimal fans in $\Gbf$ to minimal fans in the target category. 
The subspace of all minimal $\Gbf$-diagrams will be denoted
$\prob\<\Gbf\>_{\msf}$. 
In~\cite{Matveev-Asymptotic-2018} we have shown that for any fan in
$\prob$ or in $\prob\<\Gbf\>$ its minimization exists and is unique up to
isomorphism.  
\subsubsection{Tensor product}
The tensor product of two probability spaces $X=(\un X,p)$ and $Y=(\un
Y,q)$ is their independent product, $X\otimes Y:=(\un X\times\un
Y,p\otimes q)$ . For two $\Gbf$-diagrams
$\Xcal=\set{X_{i};\,\chi_{ij}}$ and
$\Ycal=\set{Y_{i};\,\upsilon_{ij}}$ we define their tensor product to
be $\Xcal\otimes\Ycal=\set{X_{i}\otimes\Ycal;\,
  \chi_{ij}\times\upsilon_{ij}}$.

\subsubsection{Constant diagrams} Given an indexing category $\Gbf$
and a probability space we can form a \term{constant} diagram $X^\Gbf$
that has all spaces equal to $X$ and all reductions equal to the identity
isomorphism. Sometimes when such constant diagram is included in a
diagram with another $\Gbf$-diagrams (such as, for example, a
reduction $\Xcal\to X^{\Gbf}$) we will write simply $X$ in place of
$X^{\Gbf}$.

\subsubsection{Entropy}
Evaluating entropy on every space in a $\Gbf$-diagram we obtain a
tuple of non-negative numbers indexed by objects in $\Gbf$, thus
entropy gives a map
\[
  \ent_{*}:\prob\<\Gbf\>\to\Rbb^{\Gbf}
\]
where the target space $\Rbb^{\Gbf}$ is a space of real-valued
functions on the set of objects in $\Gbf$ endowed with
the $\ell^{1}$-norm.
Entropy is a homomorphism in that it satisfies
\[
  \ent_{*}(\Xcal\otimes\Ycal)=\ent_{*}(\Xcal)+\ent_{*}(\Ycal)
\]

\subsubsection{Entropy distance}
Let $\Gbf$ be an indexing category and $\Kcal=(\Xcal\ot\Zcal\to\Ycal)$
be a fan of $\Gbf$-diagrams. We define the \term{entropy distance}
\[
  \kd(\Kcal)
  :=
  \left\| \ent_{*}\Zcal-\ent_{*}\Xcal \right\|_{1}
  +  \left\| \ent_{*}\Zcal-\ent_{*}\Ycal \right\|_{1}
\]
The \term{intrinsic entropy distance} between two $\Gbf$-diagrams is
defined to be the infimal entropy distance of all fans with terminal
diagrams $\Xcal$ and $\Ycal$
\[
  \ikd(\Xcal,\Ycal):=\inf\set{\kd(\Kcal)\st \Kcal=(\Xcal\ot\Zcal\to\Ycal)}
\]
The intrinsic entropy distance was introduced in \cite{Kovavcevic-Hardness-2012, Vidyasagar-Metric-2012} for probability spaces.

In~\cite{Matveev-Asymptotic-2018} it is shown that the infimum is attained,
that the optimal fan is minimal, that $\ikd$ is a pseudo-distance which
vanishes if and only if $\Xcal$ and $\Ycal$ are isomorphic and that
$\ent_{*}$ is a 1-Lipschitz linear functional with respect to $\ikd$.

\subsection{Diagrams of sets, distributions and empirical
  reductions}
\subsubsection{Distributions on sets}
For a set $S$ we denote by $\Delta S$ the collection of all
finitely-supported probability distributions on $S$. 
For a pair of
distributions $\pi_{1},\pi_{2}\in\Delta S$ we denote by 
$\left\| \pi_{1} -\pi_{2}\right\|_{1}$ the \term{total variation distance}
between them.

For a map $f:S\to S'$ between two sets we denote by
  $f_{*}:\Delta S\to\Delta S'$ the induced affine map (the map preserving convex
  combinations).

For $n\in\Nbb$ define the \term{empirical map} $\emp:S^{n}\to\Delta S$ by the
assignment below. For $\bar s=(s_{1},\dots,s_{n})\in S^{n}$ and
$A\subset S$ set
\[
\emp(\bar s)(A):=\frac1n \cdot\big|\!\set{k\st s_{k}\in A}\!\big|
\]
For a finite probability space $X=(S,p)$ the \term{empirical
  distribution} on $\Delta X$ is the push-forward
$\tau_{n}:=\emp_{*}p^{\otimes n}$. Thus
\[
\emp:X^{n}\to(\Delta X,\tau_{n})
\]
is a reduction of finite probability spaces. The construction of
empirical reduction is functorial, that is for a reduction between two
probability spaces $f:X\to Y$ the diagram of reductions
\[
\begin{cd}[row sep=small]
  X^{n}
  \arrow{r}{f^{n}}
  \arrow{d}{\emp}
  \&
  Y^{n}
  \arrow{d}{\emp}
  \\
  (\Delta X,\tau_{n})
  \arrow{r}{f_{*}}
  \&
  (\Delta Y,\tau_{n})
\end{cd}
\]
commutes.

\subsubsection{Distributions on diagrams of sets}
Let $\Set$ denote the category of sets and surjective maps.
For an indexing category $\Gbf$, we denote by $\Set\<\Gbf\>$ the category
of $\Gbf$-diagrams in $\Set$.  That is, objects in $\Set\<\Gbf\>$ are
commutative diagrams of sets indexed by $\Gbf$, the
spaces in such a diagram are sets and arrows represent
surjective maps, subject to commutativity relations.

For a diagram of sets $\Scal=\set{S_{i};\sigma_{ij}}$ we define the
\term{space of distributions
  on the diagram} $\Scal$ by
\[
\Delta\Scal
:=
\set{(\pi_{i})\in\prod_i\Delta S_{i}\st (\sigma_{ij})_{*}\pi_{i}=\pi_{j}}
\]
If $S_{0}$ is the initial set of
$\Scal$, then there is an isomorphism
\begin{align*} 
  \tageq{distributions-iso}      
  \Delta S_{0}&\oto[\cong]\Delta\Scal\\
  \Delta S_{0}\ni\pi_{0}
    &\mapsto
  \set{(\sigma_{0i})_{*}\pi_{0}}\in\Delta\Scal
  \\
  \Delta S_{0} \ni \pi_{0}
  &\leftmapsto
  \set{\pi_{i}}\in\Delta\Scal
\end{align*}

Given a $\Gbf$-diagram of sets $\Scal=\set{S_{i};\sigma_{ij}}$ and an
element $\pi\in\Delta\Scal$ we can construct a $\Gbf$-diagram of
probability spaces $(\Scal,\pi):=\set{(S_{i},\pi_{i});\sigma_{ij}}$.
Note that any diagram $\Xcal$ of probability spaces has this form.

\subsection{Conditioning}
Consider a $\Gbf$-diagram of probability spaces $\Xcal=(\Scal,\pi)$,
where $\Scal$ is a diagram of sets and $\pi\in\Delta\Scal$.  Let
$X_{0}=(S_{0},\pi_{0})$ be the initial space in $\Xcal$ and $U=X_{i}$
be another space in $\Xcal$. Since $S_{0}$ is initial, there is a map
$\sigma_{0,i}:S_{0}\to S_{i}$. Fix an atom $u\in U$ and define the
conditioned distribution $\pi_{0}(\cdot\rel u)$ on $S_{0}$ as the
distribution supported in $\sigma^{-1}_{0,i}(u)$ and for every
$s\in\sigma^{-1}_{0,i}(u)$ defined by
\[
\pi_{0}(s\rel u):=\frac{\pi_{0}(s)}{\pi_{0}(\sigma^{-1}_{0,i}(u))}
\]
Let $\pi(\cdot\rel u)\in\Delta\Scal$ be the distribution corresponding to
$\pi_{0}(\cdot\rel u)$ under the isomorphism in~(\ref{eq:distributions-iso}).
We define the \term{conditioned} $\Gbf$-diagram
$\Xcal\rel u:=(\Scal,\pi(\cdot\rel u))$.

\subsubsection{The Slicing Lemma}
In~\cite{Matveev-Asymptotic-2018} we prove the so-called Slicing Lemma that
allows to estimate the intrinsic entropy distance between two diagrams
in terms of distances between conditioned diagrams. 
Among the corollaries of the Slicing Lemma is the following
inequality.
\begin{proposition}{p:slicing}
  Let $(\Xcal\ot\hat\Xcal\to U^{\Gbf})\in\prob\<\Gbf,\Lambda_{2}\>$ be
  a fan of $\Gbf$-diagrams of probability spaces and
  $\Ycal\in\prob\<\Gbf\>$ be another diagram. Then
  \[
    \ikd(\Xcal,\Ycal)
    \leq
    \int_{U}\ikd(\Xcal\rel u,\Ycal)\d p(u) + 2\size{\Gbf}\cdot\ent U
  \]  
\end{proposition}

The fan in the assumption of the the proposition above can often be
constructed in the following manner. Suppose $\Xcal$ is a
$\Gbf$-diagram and $U=X_{\iota}$ is a space in it for some $\iota\in\Gbf$.
We can construct a fan $(\Xcal\ot[f]\hat\Xcal\to[g]
U^{\Gbf})\in\prob\<\Gbf,\Lambda_{2}\>$ by assigning $\hat X_{i}$ to be
the initial space of the (unique) minimal fan in $\Xcal$ with terminal spaces
$X_{i}$ and $U$ and $f_{i}$ and $g_{i}$ to be left and right
reductions in that fan, for any $i\in\Gbf$.

\subsection{Tropical Diagrams}\Label{s:tropical-diagrams}
A detailed discussion of the topics in this section can be found
in~\cite{Matveev-Tropical-2019}.
  
The asymptotic entropy distance between two diagrams of the same
combinatorial type is defined by
\[
  \aikd(\Xcal,\Ycal):=\lim\frac1n \ikd(\Xcal^{n},\Ycal^{n})
\]

A tropical $\Gbf$-diagram is an equivalence class of  certain
sequences of $\Gbf$-diagrams of probability spaces. Below we describe
the type of sequences and the equivalence relation.

A function $\phi:\Rbb\geq1\to\Rbb\geq0$ is called an \term{admissible
  function} if $\phi$ is non-decreasing and there is a constant
$D_{\phi}$ such that for any $t\geq1$
\[
  8t\cdot\int_{t}^{\infty}\frac{\phi(t)}{t^{2}}\d t\leq D_{\phi}\cdot\phi(t)
\]
An example of an admissible function will be $\phi(t)=t^{\alpha}$, for
$\alpha\in[0,1)$. 

A sequence $\bar\Xcal=(\Xcal(n):\,n\in\Nbb_{0})$ of diagrams of
probability spaces will be called \term{quasi-linear} with \term{defect}
bounded by an admissible function $\phi$ if it satisfies
\[
  \aikd\big(\Xcal(n+m),\Xcal(n)\otimes\Xcal(m)\big)\leq C\cdot\phi(n+m)
\]
For example for a diagram $\Xcal$, the sequence
$\bernoulli\Xcal:=(\Xcal^{n}:\, n\in\Nbb_{0})$ is $\phi$-quasi-linear for
$\phi\equiv0$ (and for any admissible $\phi$). Such sequences are
called \term{linear}.

The asymptotic entropic distance between two quasi-linear sequences
$\bar\Xcal=\big(\Xcal(n):\,n\in\Nbb_{0}\big)$ and
$\bar\Ycal=\big(\Ycal(n):\,n\in\Nbb_{0}\big)$ is defined to be 
\[
  \aikd(\bar\Xcal,\bar\Ycal):=\lim_{n\to\infty}\frac1n \ikd(\Xcal(n),\Ycal(n))
\]
and sequences are called \term{asymptotically equivalent} if
$\aikd(\bar\Xcal,\bar\Ycal)=0$. An equivalence class of a sequence
$\bar\Xcal$ will be denoted $[\Xcal]$ and the totality of all the
classes $\prob[\Gbf]$. The sum of two such equivalence classes is
defined to be the equivalence class of the sequence obtained by
tensor-multiplying representative sequences of the summands
term-wise. In addition there is a doubly transitive action of
$\Rbb_{\geq0}$ on $\prob[\Gbf]$. In~\cite{Matveev-Tropical-2019} the
following theorem is proven
\begin{theorem}{p:tropical-summary}
  Let $\Gbf$ be an indexing category. Then 
  \begin{enumerate}\def\theenumi{\roman{enumi}}
  \item 
    The space $\prob[\Gbf]$ does not depend on the
    choice of a positive admissible function $\phi$ up to isometry.
  \item 
    The space $\prob[\Gbf]$ is metrically complete.
  \item
    The map $\Xcal\mapsto\bernoulli\Xcal$ is a
    $\aikd$-$\aikd$-isometric embedding.  The space of linear
    sequences, i.e.~the image of the map above, is dense in
    $\prob[\Gbf]$.
  \item
    There is a distance-preserving homomorphism from $\prob[\Gbf]$
    into a Banach space $B$, whose image is a closed convex cone in
    $B$.
  \item The entropy functional
  \begin{align*}
  \ent_{*}:\prob[\Gbf]&\to\Rbb^{\Gbf}\\
  [\big(\Xcal(n)\big)_{n\in\Nbb_{0}}]
  &\mapsto
  \lim_{n\to\infty}
  \frac1n \ent_{*}\Xcal(n)
  \end{align*}
  is a well-defined 1-Lipschitz linear map.
  \end{enumerate}
\end{theorem}

\subsection{Asymptotic Equipartition Property for Diagrams}

Among all $\Gbf$-diagrams there is a special class of maximally
symmetric ones. We call such diagrams \term{homogeneous}, see below
for the definition. Homogeneous diagrams come very handy in many
considerations, because their structure is easier to describe then
that of general diagrams. We show below that among tropical diagrams,
those that have homogeneous representatives are dense. It means, in
particular, that when considering continuous functionals on the space
of diagrams, it suffices to only look at homogeneous diagrams.

\subsubsection{Homogeneous diagrams}

A $\Gbf$-diagram $\Xcal$ is called \term{homogeneous} if the
automorphism group $\Aut(\Xcal)$ acts transitively on every space in
$\Xcal$, by which we mean that the action is transitive on the
  support of the probability measure. Homogeneous probability spaces
are isomorphic to uniform spaces. For more complex
indexing categories this simple description is not sufficient.

\subsubsection{Tropical Homogeneous Diagrams}
The subcategory of all homogeneous $\Gbf$-diagrams will be denoted
$\Prob\<\Gbf\>_{\hsf}$ and we write $\Prob\<\Gbf\>_{\hsf,\msf}$ for
the category of minimal homogeneous $\Gbf$-diagrams. These spaces are
invariant under the tensor product, thus they are metric Abelian
monoids and the general ``tropicalization'' described
in~\cite{Matveev-Tropical-2019} can be performed.  Passing to the
tropical limit we obtain spaces of tropical (minimal) homogeneous
diagrams, that we denote by $\Prob[\Gbf]_{\hsf}$ and
$\Prob[\Gbf]_{\hsf,\msf}$, respectively.

\subsubsection{Asymptotic Equipartition Property}
In~\cite{Matveev-Asymptotic-2018} the following theorem is proven
\begin{theorem}{p:aep-complete}
  Suppose $\Xcal\in\prob\<\Gbf\>$ is a $\Gbf$-diagram of
  probability spaces for some fixed indexing category $\Gbf$.
  Then there exists a sequence
  $\bar\Hcal=(\Hcal_{n})_{n=0}^{\infty}$ of homogeneous
  $\Gbf$-diagrams such
  that
  \[\tageq{quantaep} 
    \frac{1}{n} 
    \ikd (\Xcal^{n},\Hcal_{n}) 
    \leq 
    C(|X_0|,\size{\Gbf}) \cdot \sqrt{\frac{\ln^3 n}{n}} 
  \]
  where $C(|X_0|, \size{\Gbf})$ is a constant only depending on $|X_0|$
  and $\size{\Gbf}$.
\end{theorem}

The approximating sequence of homogeneous diagrams is evidently
quasi-linear with the defect bounded by the admissible function
\[
  \phi(t)
  :=
   2C(|X_0|,\size{\Gbf})\cdot t^{3/4}
   \geq 
   2C(|X_0|,\size{\Gbf})\cdot t^{1/2}\cdot \ln^{3/2}t
\]
Thus, Theorem~\ref{p:aep-complete} above states that
$\lin(\prob\<\Gbf\>)\subset\prob[\Gbf]_{\hsf}$. On the other hand we
have shown in~\cite{Matveev-Tropical-2019}, that the space of linear
sequences $\lin(\prob\<\Gbf\>)$ is dense in $\prob[\Gbf]$. Combining
the two statements we get the following theorem.

\begin{theorem}{p:aep-tropical} For any indexing category $\Gbf$, the space 
  $\prob[\Gbf]_{\hsf}$ is dense in $\prob[\Gbf]$. Similarly, the space
  $\prob[\Gbf]_{\hsf,\msf}$ is dense in $\prob[\Gbf]_{\msf}$. 
\end{theorem}


\section{Conditioning of Tropical Diagrams}
\label{se:conditioning}
\subsection{Motivation}

Let $\Xcal\in\prob\<\Gbf\>$ be a $\Gbf$-diagram of probability
spaces containing probability space $U=X_{i_{0}}$ indexed by an
object $i_{0}\in\Gbf$.

Given an atom $u\in U$ we can define a conditioned
diagram $\Xcal\rel u$. If the diagram $\Xcal$ is homogeneous, then the
isomorphism class of $\Xcal\rel u$ is independent of $u$, so that
$\Xcal\rel u$ is a constant family. On the other hand we have shown,
that the power of any diagram can be approximated by homogeneous
diagrams, thus suggesting that in the tropical setting $\Xcal\rel U$
should be a well-defined tropical diagram, rather than a family. Below
we give a definition of tropical conditioning operation and prove its
consistency.

\subsection{Classical-tropical conditioning}
Here we define the operation of conditioning of classical diagram,
such that the result is a tropical diagram.
Let $\Xcal$ be a $\Gbf$-diagram of probability spaces and $U$ be a
space in $\Xcal$.
We define the conditioning map
\[
  [\cdot\,\rel\cdot]:
  \prob\<\Gbf\>
  \to
  \prob[\Gbf]
\]
by conditioning $\Xcal$ by $u\in U$ and averaging the corresponding
tropical diagrams:
\[
  [\Xcal\rel U]
  := 
  \int_{u\in U}\bernoulli{(\Xcal\rel u)}\d p_{U}(u)
\]
where $\bernoulli{(\Xcal\rel u)}$ is the tropical diagram represented
by a linear sequence generated by $\Xcal\rel u$,
see section~\ref{s:tropical-diagrams}.
Note that the integral on the right-hand side is just a finite convex
combination of tropical diagrams. 
Expanding all the definitions we will get for
$[\Ycal]:=[\Xcal\rel U]$ the representative sequence
\[
  \Ycal(n)
  = 
  \bigotimes_{u\in U}(\Xcal\rel u)^{\lfloor n\cdot p(u)\rfloor}
\]

\subsection{Properties}
\subsubsection{Conditioning of Homogeneous Diagrams} 
If the diagram $\Xcal$ is \emph{homogeneous}, then for any atom $u\in
U$ with positive weight
\[
  [\Xcal\rel U] 
  \,\aeq\, 
  \bernoulli{(\Xcal\rel u)}
\]

\subsubsection{Entropy}
Recall that earlier we have defined a quantity 
\[
\ent_{*}(\Xcal\rel U):=\int_{U}\ent_{*}(\Xcal\rel u)\d p_{U}(u)
\]
Now that $[\Xcal\rel U]$ is a tropical diagram, the
expression $\ent_{*}(\Xcal\rel U)$ can be interpreted in two, a priori
different, ways: by the formula above and as the entropy of the object
introduced in the previous subsection. Fortunately, the numeric value
of it does not depend on the interpretation, since the entropy is a
linear functional on $\prob[\Gbf]$.
\subsubsection{Additivity}
If $\Xcal$ and $\Ycal$ are two $\Gbf$-diagrams with $U:=X_{\iota}$,
$V:=Y_{\iota}$ for some $\iota\in\Gbf$, then
\[
  [(\Xcal\otimes\Ycal)\rel (U\otimes V)]
  =
  [\Xcal\rel U]
  +
  [\Ycal\rel V]
\]
\begin{proof}
\begin{align*}
  [(\Xcal\otimes&\Ycal)\rel (U\otimes V)]
  =
  \int_{U\otimes V}
    \bernoulli{(\Xcal\otimes\Ycal)\rel (u,v)}
  \d p(u) \d p(v)
  \\
  &=
  \int_{U\otimes V}
    (\bernoulli{\Xcal\rel u}
    +
    \bernoulli{\Ycal\rel v})
  \d p(u) \d p(v)
  =
  \int_{U}
    \bernoulli{\Xcal\rel u}
  \d p(u)
  +
  \int_{V}
    \bernoulli{\Ycal\rel v}
  \d p(v)
  \\
  &=
  [\Xcal\rel U]
  +
  [\Ycal\rel V]
\end{align*}
\end{proof}

\subsubsection{Homogeneity}\Label{s:cond-homo}
It follows that for any diagram $\Xcal$ with a space $U$ and
$n\in\Nbb_{0}$ holds
\[
  [\Xcal^{n}\rel U^{n}] 
  =
  n\cdot [\Xcal\rel U]
\]

\subsection{Continuity and Lipschitz property}
\begin{proposition}{p:cond-lip}
  Let $\Gbf$ be a complete poset category,
  $\Xcal,\Ycal\in\prob\<\Gbf\>$ be two $\Gbf$ diagrams, $U:=X_{\iota}$
  and $V:=Y_{\iota}$ be two spaces in $\Xcal$ and $\Ycal$,
  respectively, indexed by some $\iota\in\Gbf$.
  Then
  \[
    \aikd\Big([\Xcal\rel U],\,[\Ycal\rel V]\Big)
    \leq
    (2\cdot\size{\Gbf}+1)\cdot\ikd\left(\Xcal,\Ycal\right)
  \]
\end{proposition}
Using homogeneity property of conditioning, Section~\ref{s:cond-homo},
we can obtain the following stronger inequality.
\begin{corollary}{p:cond-lip-aikd}
  In the setting of Proposition~\ref{p:cond-lip} holds
  \[
    \aikd\Big([\Xcal\rel U],\,[\Ycal\rel V]\Big)
    \leq
    (2\cdot\size{\Gbf}+1)\cdot\aikd\left(\Xcal,\Ycal\right)
  \]
\end{corollary}

Before we prove Proposition~\ref{p:cond-lip} we will need some
preparatory lemmas. 

\begin{lemma}{p:dist-cond-types}
Let $\Acal$ be a $\Gbf$-diagram of probability spaces and $E$ be a
space in it.  Let $\qbf: E^{n}\to(\Delta E,\tau_{n})$ be the empirical
reduction. Then for any $n\in\Nbb$ and any $\bar {e},\bar {e}'\in
E^{n}$
\[  
  \ikd(\Acal^{n}\rel \bar{e},\Acal^{n}\rel \bar{e}')
  \leq 
  {n}\cdot\|\ent_{*}(\Acal)\|_{1}\cdot\|\qbf(\bar{e})-\qbf(\bar{e}')\|_{1}
\]  
\end{lemma}
\begin{proof}
 To prove the lemma we construct a coupling between $\Acal^{n}\rel
 \bar{e}$ and $\Acal^{n}\rel \bar{e}'$ in the following manner. Note
 that there exists a permutation $\sigma\in S_{n}$ such that
\[
  \big|\!\set{i\st e_{i}\neq e_{\sigma i}'}\!\big| 
  =
  \frac{n}{2}\cdot\|\qbf(\bar e)-\qbf(\bar e')\|_{1}
\]
Let 
\begin{align*}
  I
  &=
  \set{i \st e_{i}=e'_{\sigma i}}
  \\
  \tilde I
  &=
  \set{i \st e_{i}\neq e'_{\sigma i}}  
\end{align*}
Using that $|\tilde I|=\frac{n}{2}\cdot\|\qbf(\bar e)-\qbf(\bar
e')\|_{1}$ we can estimate
\begin{align*}
  \ikd\Big(\Acal^{n}\rel\bar e\,,\,\Acal^{n}\rel\bar e'\Big)
  &=
  \ikd\left(
    \bigotimes_{i=1}^{n}(\Acal\rel e_{i})
    \,,\,
    \bigotimes_{i=1}^{n}(\Acal\rel e'_{\sigma i})\right)
  \\
  &\leq
  \sum_{i\in I}
    \kd(\Acal\rel e_{i}\ootoo[=]\Acal\rel e'_{\sigma i})
  \,+\,  
  \sum_{i\in\tilde I}
    \kd(\Acal\rel e_{i}\ootoo[\otimes]\Acal\rel e'_{\sigma i})
  \\
  &\leq
  n\cdot\|\ent_{*}(\Acal)\|_{1}\cdot\|\qbf(\bar e)-\qbf(\bar e')\|_{1}
\end{align*}
where $\Acal\oto[=]\Bcal$ denotes the isomorphism coupling of two
naturally isomorphic diagrams, while $\Acal\oto[\otimes]\Bcal$
denotes the ``independence'' coupling.
\end{proof}

\begin{lemma}{p:int-dist-cond}
Let $\Acal$ be a $\Gbf$-diagram of probability spaces and $E$ be a
space in $\Acal$. Then
\[  
  \int_{E^{n}}
  \ikd(\Acal^{n},\Acal^{n}\rel\bar e)\d p(\bar{e})
  \leq
  2n\cdot\size{\Gbf}\cdot\ent(E) + \o(n)
\]  
\end{lemma}
\begin{proof}
First we apply Proposition~\ref{p:slicing} slicing the first argument
\begin{align*}
  \int_{E^{n}}
    &
    \ikd(\Acal^{n},\Acal^{n}\rel\bar e)
  \d p(\bar e)
  \\
  &\leq
  \int_{E^{n}}
  \int_{E^{n}}
      \ikd(\Acal^{n}\rel\bar e',\Acal^{n}\rel\bar e)
      \d p(\bar e') \d p(\bar e)
  +
  2n\cdot\size{\Gbf}\cdot\ent(E) 
\end{align*}
We will argue now that the double integral on the right-hand side
grows sub-linearly with $n$.  We estimate the double integral by
applying Lemma \ref{p:dist-cond-types} to the integrand
\begin{align*}
  \int_{E^{n}}
  \int_{E^{n}}
  &
      \ikd(\Acal^{n}\rel\bar e',\Acal^{n}\rel\bar e)
      \d p(\bar e') \d p(\bar e)
  \\    
  &\leq
  \int_{E^{n}}
  \int_{E^{n}}
      n\cdot
      \size{\Gbf}\cdot
      |\ent_{*}(\Acal)|_{1}\cdot
      |\qbf(\bar e)-\qbf(\bar e')|_{1}
  \d p(\bar e') \d p(\bar e)
  \\    
  &=
   n\cdot
  \size{\Gbf}\cdot
  |\ent_{*}(\Acal)|_{1}\cdot\int_{\Delta E}
  \int_{\Delta E}     
      |\pi-\pi'|_{1}
  \d \tau_{n}(\pi)
  \d \tau_{n}(\pi')
  =
  \o(n)
\end{align*}
where the convergence to zero of the last double integral follows from
Sanov's theorem.  
\end{proof}

\begin{corollary}{p:dist-cond}
  Let $\Acal$ be a $\Gbf$-diagram and $E$ a probability space
  included in $\Acal$. Then
  \[
    \aikd\Big(\bernoulli{\Acal},[\Acal\rel E]\Big)
    \leq 
    2\size{\Gbf}\cdot\ent(E)
  \]
\end{corollary}
\begin{proof}
Let $n\in\Nbb$. Then
\begin{align*}
  \aikd\Big(\bernoulli{\Acal},[\Acal\rel E]\Big)
  &=
  \frac{1}{n}
  \aikd\Big(\bernoulli{\Acal^{n}},[\Acal^{n}\rel E^{n}]\Big)
  \\
  &=
  \frac1n
  \aikd\left(
     \bernoulli{\Acal^{n}},
     \int_{E^{n}}\bernoulli{\Acal^{n}\rel\bar e}\d p(\bar e)
  \right)
  \\
  &\leq
  \frac1n
  \int_{E^{n}}
    \aikd\left(\bernoulli{\Acal^{n}},\,\bernoulli{\Acal^{n}\rel\bar e}\right)
  \d p(\bar e)
  \\
  &=
  \frac1n
  \int_{E^{n}}
  \aikd(\Acal^{n},\Acal^{n}\rel\bar e)
  \d p(\bar e)
  \\
  &\leq
  2\cdot\size{\Gbf}\cdot
  \ent(E)
  +
  \o(n^0)
\end{align*}
where we used Lemma \ref{p:int-dist-cond} and the fact that
$\aikd\leq\ikd$ in the last line.  We finish the proof by taking the
limit $n\to\infty$.
\end{proof}

\begin{proof}[of Proposition~\ref{p:cond-lip}] 
We start with a note on general terminology: a reduction $f:A\to B$ of
probability spaces can also be considered as a fan $\Fcal:=(A\ot[=]
A\to[f]B)$. Then the entropy distance of $f$ is
\[
  \kd(f):=\kd(\Fcal)=\ent A-\ent B
\]
If the reduction $f$ is a part of a bigger diagram containing also
space $U$, then the following inequality holds
\[
\int_{U}\kd(f\rel u)\d p(u)\leq \kd(f)
\]

Let $\Kcal\in\prob\<\Gbf,\Lambdabf_{2}\>$ 
\[
\Kcal=
\left(
\begin{cd}[row sep=0mm,column sep=small]
\Xcal
\&
\Zcal
\arrow{l}[above]{f}
\arrow{r}{g}
\&
\Ycal
\end{cd}
\right)
\in\prob\<\Gbf,\Lambdabf_{2}\>=\prob\<\Lambdabf_{2},\Gbf\>
\]
be an optimal coupling between $\Xcal$ and $\Ycal$.
It can also we viewed as a $\Gbf$-diagram of two-fans,
$\Kcal=\set{\Kcal_{i}}_{i\in\Gbf}$ each of which is a minimal coupling
between $X_{i}$ and $Y_{i}$. Among them is the minimal fan 
$\Wcal:=\Kcal_{\iota}=(U\oot[f_{\iota}] W\too[g_{\iota}] V)$.

We use triangle inequality to bound the distance $\aikd\Big([\Xcal\rel
  U],[\Ycal\rel V]\Big)$ by four summands as follows.
\begin{align*}
  \aikd\Big([\Xcal\rel U],[\Ycal\rel V]\Big)
  \leq&
  \aikd\Big([\Xcal\rel U],[\Zcal\rel U]\Big) \;+
  \aikd\Big([\Zcal\rel U],[\Zcal\rel W]\Big) \,+
  \\&
  \aikd\Big([\Zcal\rel W],[\Zcal\rel V]\Big) +
  \aikd\Big([\Zcal\rel V],[\Ycal\rel V]\Big)
\end{align*}
We will estimate each of the four summands separately. The bound for
the first one is as follows.
\begin{align*}
  \aikd\Big(&[\Xcal\rel U],[\Zcal\rel U]\Big)
  =
  \aikd
  \left(
    \int_{U}
      \bernoulli{\Xcal\rel u}
    \d p(u)
  ,
    \int_{U}
      \bernoulli{\Zcal\rel u}
    \d p(u)
  \right)
  \\\label{eq:cond-case-1-triang}
  &\leq
  \int_{U}\aikd\left(\bernoulli{\Xcal\rel u},\,
               \bernoulli{\Zcal\rel u}\right)\d p(u)  
  =
  \int_{U}\aikd\left(\Xcal\rel u,\,
  \Zcal\rel u\right)\d p(u)
  \\
  &
  \leq
  \int_{U}\ikd\left(\Xcal\rel u,\,
  \Zcal\rel u\right)\d p(u)
  \leq
  \int_{U}\kd(f\rel u)\d p(u)
  \\
  &\leq
   \sum_{i\in\Gbf}\int_{U}\kd(f_{i}\rel u)\d p(u)
   =
   \sum_{i\in\Gbf}\kd(f_{i})
   =
   \kd(f)
\end{align*}
An analogous calculation shows that  
\[
  \aikd\Big([\Zcal\rel V],[\Ycal\rel V]\Big)
  \leq
  \kd(g)
\]

To bound the second summand we will use Corollary~\ref{p:dist-cond}
\begin{align*}
  \aikd\Big([\Zcal\rel U],[\Zcal\rel W]\Big)
  &=
  \aikd
  \left(
    \int_{U}\bernoulli{\Zcal\rel u}\d p(u)
    ,
    \int_{W} \bernoulli{\Zcal\rel w} \d p(w)
  \right)
  \\
  &=
  \aikd
  \left(
    \int_{U}\bernoulli{\Zcal\rel u} \d p(u)
    ,
    \int_{U}\int_{W\rel u}\bernoulli{\Zcal\rel w}\d p(w|u)\d p(u)
  \right)
  \\
  &\leq 
  \int_{U}
    \aikd
    \left(
      \bernoulli{\Zcal\rel u},\int_{W\rel u}\bernoulli{\Zcal\rel w} \d p(w|u)
    \right)
  \d p(u)
\end{align*}

We will now use Corollary \ref{p:dist-cond} with $\Acal = \Zcal\rel u$
and $E = W \rel u$ to estimate the integrand. Then,
\begin{align*}
  \aikd\Big([\Zcal\rel U],[\Zcal\rel W]\Big)
  &=
  \int_{U}
  \aikd\left(
    \bernoulli{\Zcal\rel u},\int_{W\rel u}\bernoulli{\Zcal\rel w} \d p(w|u)
  \right)
  \d p(u) 
  \\
  &\leq 
  2\size{\Gbf} \cdot \int_U \ent( W \rel u ) \d p(u) 
  \\
  &\leq 
  2\size{\Gbf} \cdot \ent( W \rel U ) 
  \leq 
  2\size{\Gbf} \cdot \kd(f)
\end{align*}
Similarly
\[
  \aikd\Big([\Zcal\rel W],[\Zcal\rel V]\Big)
  \leq
  2\size{\Gbf} \cdot \kd(g)
\]
Combining the estimates we get
\[
  \aikd\Big([\Xcal\rel U],[\Ycal\rel V]\Big)
  \leq 
  (2\size{\Gbf}+1)\cdot(\kd{(f)}+\kd(g))
  =
  (2\size{\Gbf}+1)\cdot\ikd(\Xcal,\Ycal)
\]
\end{proof}

\subsection{Tropical conditioning}
Let $[\Xcal]$ be a tropical $\Gbf$-diagram and $[U]=[X_{\iota}]$ for
some $\iota\in\Gbf$. Choose a representative 
$\big(\Xcal(n)\big)_{n\in\Nbb_{0}}$ and denote $u(n):=X_{\iota}(n)$.
We define now a conditioned diagram $[\Xcal\rel
  U]$ by the following limit
\[
  [\Xcal\rel U]:=\lim_{n\to\infty}\frac1n [\Xcal(n)\rel U(n)]
\]
Proposition~\ref{p:cond-lip-aikd} guarantees, that the limit exists and
is independent of the choice of representative.
For a fixed $\iota\in\Gbf$ the conditioning is a linear Lipschitz map
\[
  [\,\cdot\;\rel\;\cdot_{\iota}\,]:\prob[\Gbf]\to\prob[\Gbf]
\]

\bibliographystyle{alpha}       
\bibliography{ReferencesCond}
\end{document}